\documentclass[11pt]{amsart}
\usepackage{amsmath,amssymb,amsthm}
\usepackage{subcaption,graphicx}
\usepackage[hidelinks]{hyperref}
\usepackage{microtype}
\newtheorem{definition}{Definition}[section]
\newtheorem{theorem}[definition]{Theorem}
\newtheorem{proposition}[definition]{Proposition}
\newtheorem{lemma}[definition]{Lemma}
\newtheorem{corollary}[definition]{Corollary}
\newcommand{\ph}{\varphi}
\newcommand{\AP}{X_{\mathrm{AP}}}
\newcommand{\SB}{X_{\mathrm{S}}}
\newcommand{\one}{\mathbf{1}}
\newcommand{\cl}{\operatorname{cl}}
\newcommand{\st}{\star}
\newcommand{\subfigurerowgap}{\par\vspace{10pt}}
\hypersetup{
  pdftitle={A Recognizable Substitution for the Anti-Penrose Tiling},
  pdfauthor={Miki Imura}
}
\graphicspath{{figures/}}
\title[Substitution for the Anti-Penrose tiling]{
  A Recognizable Substitution for the Anti-Penrose Tiling}
\author{Miki Imura}
\date{October 22, 2026 --- version 2}
\begin{document}
\begin{abstract}
We present a substitution on six marked rhombs, called Seabed,
with inflation factor $\ph^3$.
We prove that its tiling space is mutually locally derivable with the unmarked
Anti-Penrose pentagrid hull. Eleven vertex patches of depth at most four
recognize the parent vertices: their acceptance domains coincide with the
five internally contracted windows. Matching the fine tiles around each
parent rhomb to a replacement patch recovers their markings uniquely.
The inverse local maps have a uniform finite radius
and extend to all singular limit tilings in the hull. The construction also
gives a recognizable parent hierarchy.
\end{abstract}
\maketitle

\section{Introduction}
Penrose rhomb tilings~\cite{Penrose1979} can be obtained by dualizing a
pentagrid~\cite{deBruijn1981}. Changing the sum of the grid offsets changes
the local-isomorphism class. The sum $1/2$ modulo one gives the
\emph{Anti-Penrose} class~\cite[Section 3.6]{GrimmSchreiber2002}, distinct
from the usual Penrose class with sum zero.

\begin{figure}[htbp]
  \centering
  \includegraphics[width=\textwidth]{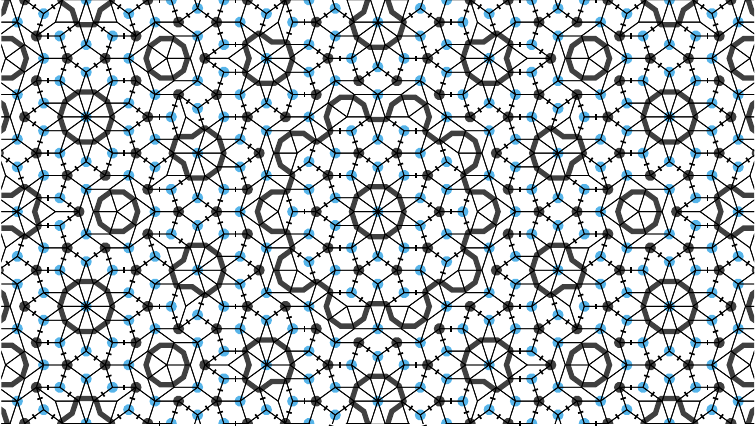}
  \caption{A patch of the Seabed tiling.}
  \label{fig:intro}
\end{figure}

We give a substitutive description using six locally recoverable
marked rhomb types. We call the resulting marked tiling \emph{Seabed}
(Figure~\ref{fig:intro}). Its inflation factor is $L=\ph^3$, where
$\ph=(1+\sqrt5)/2$.
For context, Lutfalla's 2021 dissertation states of the Anti-Penrose tiling:
``To our knowledge it is unknown whether it is a substitution tiling or
not''~\cite[Section 6.2.2, p.~167]{Lutfalla2021}.

Our main result is mutual local derivability (MLD)~\cite{BaakeSchlottmannJarvis1991}
between the Seabed substitution space and the unmarked Anti-Penrose hull.
One map erases the markings. The inverse recognizes parent vertices from
eleven unmarked patches, reconstructs the parent rhombs, and obtains the
markings from the six replacement rules.
Recovering the unit-tile markings and recognizing the larger parent tiling
are distinct consequences of this construction.

\section{The Seabed substitution}
\label{sec:substitution}
The six marked prototiles are unit-edge rhombs: three thin types
(acute angle $36^\circ$) and three thick types (acute angle $72^\circ$).
Figure~\ref{fig:prototiles} gives equivalent arrow and graphical markings;
the circled reference corner fixes their placement.
Rotations by $36^\circ$ and horizontal reflection are allowed.

\begin{figure}[htbp]
  \centering

  \begin{subfigure}{\dimexpr\textwidth/3\relax}
    \centering
    \includegraphics{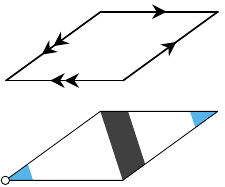}
    \caption{Thin-I}
  \end{subfigure}%
  \begin{subfigure}{\dimexpr\textwidth/3\relax}
    \centering
    \includegraphics{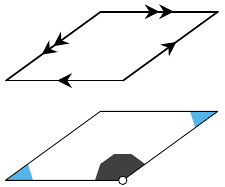}
    \caption{Thin-II}
  \end{subfigure}%
  \begin{subfigure}{\dimexpr\textwidth/3\relax}
    \centering
    \includegraphics{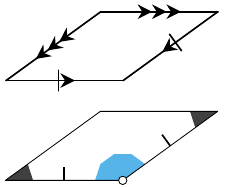}
    \caption{Thin-III}
  \end{subfigure}

  \subfigurerowgap

  \begin{subfigure}{\dimexpr\textwidth/3\relax}
    \centering
    \includegraphics{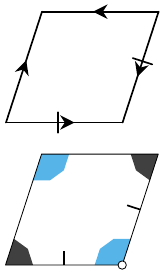}
    \caption{Thick-I}
  \end{subfigure}%
  \begin{subfigure}{\dimexpr\textwidth/3\relax}
    \centering
    \includegraphics{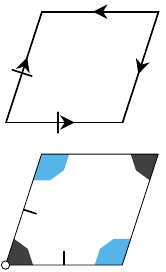}
    \caption{Thick-II}
  \end{subfigure}%
  \begin{subfigure}{\dimexpr\textwidth/3\relax}
    \centering
    \includegraphics{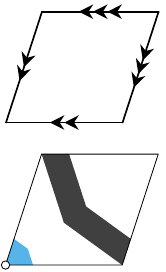}
    \caption{Thick-III}
  \end{subfigure}

  \caption{The six types, with directed edge markings above and graphical markings
    below. Open circles indicate reference corners and are not part of the markings.}
  \label{fig:prototiles}

\end{figure}

The six rules in Figure~\ref{fig:substitution} replace a marked parent
at scale $L$ by unit-edge children: 17 for a thin parent and 25 for a thick
parent. Some children cross the straight parent boundary. Apply the rules
to all parents and identify coincident children, including their markings.

\begin{figure}[htbp]
  \centering

  \begin{subfigure}{\dimexpr\textwidth/3\relax}
    \centering
    \includegraphics{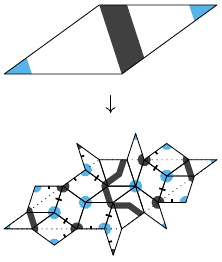}
    \caption{Thin-I}
  \end{subfigure}%
  \begin{subfigure}{\dimexpr\textwidth/3\relax}
    \centering
    \includegraphics{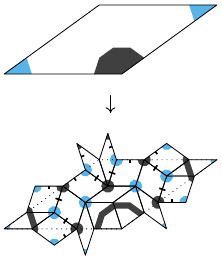}
    \caption{Thin-II}
  \end{subfigure}%
  \begin{subfigure}{\dimexpr\textwidth/3\relax}
    \centering
    \includegraphics{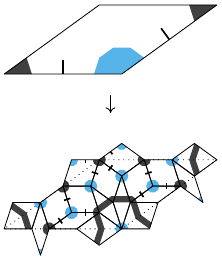}
    \caption{Thin-III}
  \end{subfigure}

  \subfigurerowgap

  \begin{subfigure}{\dimexpr\textwidth/3\relax}
    \centering
    \includegraphics{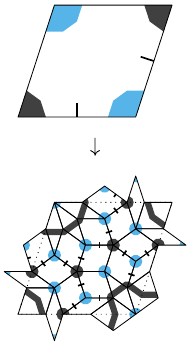}
    \caption{Thick-I}
  \end{subfigure}%
  \begin{subfigure}{\dimexpr\textwidth/3\relax}
    \centering
    \includegraphics{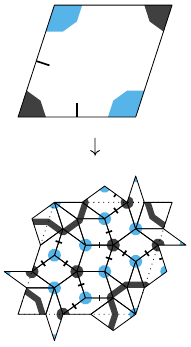}
    \caption{Thick-II}
  \end{subfigure}%
  \begin{subfigure}{\dimexpr\textwidth/3\relax}
    \centering
    \includegraphics{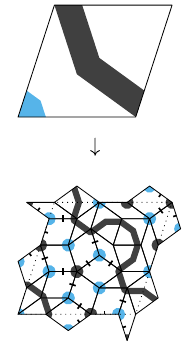}
    \caption{Thick-III}
  \end{subfigure}

  \caption{The six Seabed replacement patches. A parent (top) specifies the marked child
    patch (bottom). Children crossing a parent boundary belong to both adjacent
    replacement patches and are identified when these patches are joined.}
  \label{fig:substitution}

\end{figure}

\begin{definition}
\label{def:seabed}
The Seabed space $\SB$ consists of the plane tilings whose finite marked
patches occur in iterates of the six rules on a marked prototile.
\end{definition}
We do not claim that every tiling satisfying the edge-matching rules
belongs to $\SB$.

\section{The Anti-Penrose windows}
\label{sec:objects}
Put $v_j=(\cos(2\pi j/5),\sin(2\pi j/5))$ and $\pi(k)=\sum_jk_jv_j$.
The pentagrid has lines $x\cdot v_j=n+\gamma_j$, with
$\sum_j\gamma_j=1/2\pmod1$. Its cells have labels
$k_j=\lceil x\cdot v_j-\gamma_j\rceil$ and are dual to vertices $\pi(k)$;
two-line intersections are dual to rhombs.
We define $\AP$ as the local closure of the regular pentagrid rhomb tilings
and their translates. Thus it includes singular rhomb resolutions obtained
as local limits. Convergence means agreement on growing balls after
translations tending to zero.

Write $q=k+\gamma-\tfrac12\one$. Modulo $\mathbb Z\one$, choose the
height $h=\sum_jq_j$ in $\{-2,-1,0,1,2\}$; brackets $[\ ]_5$ denote
this normalization. In the internal basis $(1,\zeta^2)$,
$\zeta=e^{2\pi i/5}$, the generators are
\begin{equation}
  \label{eq:generators}
  (G_0,G_1,G_2,G_3,G_4)
  =((1,0),(0,1),(-1,-\ph),(\ph,\ph),(-\ph,-1)).
\end{equation}
The internal point is $q^\st=\sum_jq_jG_j$, and the five windows are
\begin{equation}
  \label{eq:windows}
  W_h=\left\{\sum_jc_jG_j:\ -\tfrac12\leq c_j\leq\tfrac12,
  \quad\sum_jc_j=h\right\}.
\end{equation}
Indeed $c_j=q_j-x\cdot v_j$ satisfies these strip constraints and has
internal projection $q^\st$. The windows are shown in
Figure~\ref{fig:windows}. Closed windows describe the geometry; boundary
choices will be handled by local limits.

We first work at \emph{generic} phases, avoiding the grid and rule-domain
boundaries and their images under every inflation iterate. These exclude
only countably many proper affine sets and form a dense family. Every
finite regular unmarked patch persists at a nearby generic phase.

\begin{lemma}
\label{lem:edges}
At a generic vertex of height $h$, an edge $+v_j$ occurs exactly when
$q^\st\in W_h\cap(W_{h+1}-G_j)$; for $-v_j$ use
$W_h\cap(W_{h-1}+G_j)$. Here windows outside $[-2,2]$ are empty.
\end{lemma}
\begin{proof}
The cells $k$ and $k+e_j$ satisfy the same strict strip inequalities except
in family $j$, where their intervals are consecutive. A segment joining
interior points stays inside the other strips and crosses their common
$j$-line, so the cells share an edge. The converse is immediate.
The translated intersections that would wrap height $2$ to $-2$, or
vice versa, have empty interiors by~\eqref{eq:windows}.
\end{proof}

\subsection{The parent windows}
Let $M$ be the circulant matrix whose first row is $(2,1,-1,-1,1)$.
Its physical and
internal actions are multiplication by $L=\ph^3$ and
$\alpha=3-2\ph=-\ph^{-3}$, respectively; its row sum is 2.
The parent-vertex windows in fine-tiling coordinates are therefore
\begin{equation}
  \label{eq:parent-window}
  U_h=\alpha W_{[3h]_5}.
\end{equation}
Thus the source layers for $h=-2,-1,0,1,2$ are $-1,2,0,-2,1$.
For the phase relation, put $N=M-4I$. Since $MN=NM=I-J$, where every
entry of $J$ is one, $N$ is the inverse on
$\mathbb Z^5/\langle\one\rangle$. The parent is the pentagrid at
\begin{equation}
  \label{eq:parent-phase}
  \gamma_{\rm p}=\tfrac12\one+N(\gamma-\tfrac12\one),
\end{equation}
scaled by $L$. Indeed $Nq$ has internal coordinate $q^\st/\alpha$ and
height $[3h]_5$, and $\pi(Nk)=L^{-1}\pi(k)$.

\begin{figure}[htbp]
  \centering
  \input{figures/windows-parent-panels}
  \caption{The five windows $W_h$ at one scale. Thin black lines delimit
  oriented vertex-star regions; the light gray region is
  $U_h=-\ph^{-3}W_{[3h]_5}$, outlined in bold.}
  \label{fig:windows}
\end{figure}

\section{Recognizing the parent vertices}
\label{sec:stars}
The \emph{$k$-star} at a vertex is the embedded patch formed by all tiles
incident to vertices at edge-graph distance less than $k$.

Let $P_1,\ldots,P_{11}$ be the patches in Figure~\ref{fig:atlas}, matched
up to rotations by $36^\circ$ and reflection, without markings or grid data.

\begin{proposition}
\label{thm:vertices}
A generic vertex is a parent vertex if and only if its neighborhood
matches one of $P_1,\ldots,P_{11}$. Exactly one entry matches, counting
symmetric placements of an entry once. A Euclidean radius 5 suffices.
\end{proposition}
\begin{proof}
Lemma~\ref{lem:edges} partitions each $W_h$ into polygons of constant
vertex star. At a displacement $a\in\mathbb Z^5$, translate these polygons
by $-a^\st$ and use height $h+\sum_ja_j$. Intersecting along the actual
edge neighbors gives the acceptance domains of any finite rooted patch.
Both present and absent edges are recorded. Grouping all components and
heights with the same embedded patch removes the unobserved internal labels.

Let $A_{h,r}$ be the resulting domain for entry $P_r$, including all its
allowed orientations. The exact polygon identities are
\begin{equation}
  \label{eq:atlas-partition}
  \bigcup_{r=1}^{11} A_{h,r}=U_h,\qquad
  \operatorname{int}A_{h,r}\cap\operatorname{int}A_{h,s}=\varnothing
  \quad(r\ne s).
\end{equation}
Figure~\ref{fig:window-union} shows the union. We verify~\eqref{eq:atlas-partition}
by computer using exact polygon arithmetic over $\mathbb Q(\ph)$. Each selected cell
lies in $U_h$, and the cells cover $U_h$. Subdividing collinear edges at
their endpoints cancels the internal edges; the remaining directed
segments are exactly the boundary of $U_h$. No cell matches two entries,
and no occurrence of a listed patch lies outside $U_h$, including
occurrences in other components or layers.

At generic positions~\eqref{eq:atlas-partition} proves the assertion.
A 4-star lies within radius $3+2=5$: its centers have distance at most 3
and rhomb diameters are less than 2.
\end{proof}

\begin{figure}[htbp]
  \centering
  \input{figures/star-atlas-panels}
  \caption{The eleven unmarked patches. The open circle is the root;
  solid points mark the other centers whose complete stars are prescribed.
  The label $k$ gives the star depth.}
  \label{fig:atlas}
  \vspace{20pt}
\end{figure}
\begin{figure}[!htbp]
  \centering
  \input{figures/window-union-panels}
  \caption{The union of the eleven patch domains in each layer, enlarged
  from Figure~\ref{fig:windows}. The thick boundary is $U_h$; thin black lines
  separate central vertex stars. The union equals the parent window in every layer.}
  \label{fig:window-union}
\end{figure}
\clearpage
\section{From the parent windows to the substitution}
\label{sec:six}
\subsection{The unmarked hierarchy}
\begin{lemma}
\label{lem:parents}
Join selected vertices whose difference is $\pm Lv_j$. This gives exactly
the parent pentagrid graph. Its rhombs are the four selected vertices
$a,a+Lv_i,a+L(v_i+v_j),a+Lv_j$ for $i<j$.
\end{lemma}
\begin{proof}
The selected vertex set is that of the scaled pentagrid
\eqref{eq:parent-phase}, so Lemma~\ref{lem:edges} applies at scale $L$.
A four-edge cycle bounds a single rhomb: thick and thin areas have ratio
$\ph$, and a nonnegative integer combination of these areas equals either
one only for that single tile.
\end{proof}

The identity $U_h=\alpha W_{[3h]_5}$ now gives a hierarchy. After
rescaling by $L^{-1}$, the parent is again an Anti-Penrose tiling, with
phase~\eqref{eq:parent-phase}. The same construction therefore applies
at every subsequent level. The inverse phase map, given by $M$, also
supplies a finer level. Thus a generic tiling belongs to a hierarchy of
unmarked rhomb tilings in both directions. It remains to identify the
subdivisions with the six rules and check their markings across levels.

\subsection{Matching the replacement rules}
For each parent rhomb, compare the unmarked children with all placements
of the six replacement patches, including boundary-crossing children.

\begin{proposition}
\label{prop:six}
Each generic parent matches a unique marked rule, up to symmetries fixing
the marked patch. Its children are fine rhombs with disjoint interiors
covering the parent.
\end{proposition}
\begin{proof}
Choose a parent root whose incident edges are $+Lv_i,+Lv_j$, with $i<j$.
For each of the ten orientation pairs $(i,j)$ and five root heights $h$, let
$D_{ijh}$ be the intersection of the four translated parent windows.
For a placed rule $r$, further intersect with the translated fine windows
of all its child corners, obtaining $D_{ijh,r}$. Empty interiors are
omitted. The domain and patch checks in this proof use exact computer
arithmetic in $\mathbb Q(\ph)$.
The exact intersections give 30 parent domains and 60 rule domains, with
\begin{equation}
  \label{eq:six-partition}
  D_{ijh}=\bigcup_rD_{ijh,r},\qquad
  \operatorname{int}D_{ijh,r}\cap\operatorname{int}D_{ijh,s}
  =\varnothing\quad(r\ne s).
\end{equation}
All placements are included; equivalent symmetries are counted once.
Each $D_{ijh,r}$ is contained in $D_{ijh}$ by construction. We compute
the polygon areas and pairwise intersection areas in $\mathbb Q(\ph)$:
the former sum to the area of $D_{ijh}$ and the latter vanish.
The finite union is closed; any point of $D_{ijh}$ left uncovered would
therefore leave a positive-area gap. This proves~\eqref{eq:six-partition}.

Each child also satisfies the three remaining strip inequalities at its
two-line grid intersection. These are affine inequalities in the root's
internal coordinate, so checking the vertices of its rule domain proves
that it is a fine-grid rhomb throughout the domain interior.

For the physical patch, pairwise child intersection areas vanish, and
the areas of the children clipped to the parent sum to the parent's area.
Closedness again gives coverage. We use physical oblique coordinates, obtained by
the conjugation $\ph\mapsto1-\ph$ in~\eqref{eq:generators}; their common
area factor cancels from the comparison.
\end{proof}

\subsection{Compatibility between levels}
Each rhomb in the hierarchy receives the marked type selected by its
subdivision at the level below. We show that a parent rule assigns exactly
these markings to its children.

For a placed rule $r$ in layer $h$, write $D_r=D_{ijh,r}$ and put
\[
  C_r=\alpha^{-1}D_r.
\]
This is its occurrence domain in the normalized parent tiling, at height
$[3h]_5$. The index $r$ includes the orientation and the placement of the
marking. If a child has root displacement $b\in\mathbb Z^5$ from the
parent root in fine coordinates, its internal coordinate is
$\alpha x+b^\st$ when the parent's is $x$.

\begin{lemma}
\label{lem:compatibility}
For every child of every placed rule $r$, with prescribed marked
placement $s$, the corresponding occurrence domains satisfy
\begin{equation}
  \label{eq:domain-compatibility}
  \alpha C_r+b^\st\subseteq C_s.
\end{equation}
Consequently the markings agree between consecutive hierarchy levels
and on children shared by different parents.
\end{lemma}
\begin{proof}
Use the child's positive-edge root and orientation to choose $C_s$;
its height is $[h+\sum_jb_j]_5$. For the 60 parent rule domains, there
are 1,260 child occurrences. Exact computer arithmetic in $\mathbb Q(\ph)$
verifies that, in every case, the vertices of
$\alpha C_r+b^\st=D_r+b^\st$ satisfy every defining half-plane
inequality of $C_s$. Convexity gives~\eqref{eq:domain-compatibility}.

The map $x\mapsto\alpha x+b^\st$ is invertible, so it sends the
interior of $C_r$ into the interior of $C_s$. By the uniqueness in
Proposition~\ref{prop:six}, the child's own subdivision therefore selects
the prescribed marking $s$. A shared child has just one such subdivision,
so both parents assign it the same marking. The argument applies at
every level of the hierarchy.
\end{proof}

Let $Y_0$ be the generic Anti-Penrose tilings with these markings and
$Y=\cl Y_0$. Proposition~\ref{prop:six} and Lemma~\ref{lem:compatibility}
show that replacement and taking a marked parent preserve $Y_0$.

\section{Mutual local derivability}
\label{sec:mld}
\begin{theorem}
\label{thm:mld}
The spaces $\AP$ and $\SB$ are mutually locally derivable. Erasure $E$
and marking recovery $F$ are inverse maps. A recovery radius of $R=4\ph^3+4$ fine
edge lengths works throughout the hull, including singular limit tilings.
\end{theorem}
\begin{proof}
For a generic tiling, Proposition~\ref{thm:vertices} and
Lemma~\ref{lem:parents} recover the parents, and
Proposition~\ref{prop:six} supplies the markings, consistently by
Lemma~\ref{lem:compatibility}. Every fine rhomb is
contained in a replacement patch, since the parent rhombs tile the plane
and their child patches cover them.

To bound the radius, fix a point in the fine rhomb to be marked.
A parent whose replacement patch contains that rhomb meets it in positive
area, so its vertices
are at distance less than $2L+2$. Their detection requires another 5.
Every template child meets the parent and has diameter less than 2;
all data for matching it lie within $4L+4$. Since $2L+7<4L+4=R$,
the construction is a fixed local rule, while
erasure has radius zero.

If generic tilings converge locally to $T\in\AP$, each fixed radius-$R$
patch eventually agrees, after a small translation, with the patch in
$T$. The same local rule therefore defines $F(T)$ independently of the
approximating sequence. Marked patches agree on overlaps because any
finite disagreement would already occur in a generic approximant.
Consequently $F:\AP\to Y$ and $E:Y\to\AP$ remain inverse.
This includes all limiting rhomb resolutions in the definition of $\AP$;
it introduces no independent choices of closed-window boundary membership.

It remains to identify $Y$ with $\SB$. The generic marked grids have one
common repetitive language: each finite marked patch imposes strict affine
window inequalities with nonempty relative interior, and the internal
lattice projection is dense in each height class. A finite cover of the
compact windows by translates of such a patch domain gives bounded return
distances. Each of the six types occurs.
The compatible hierarchy of Section~\ref{sec:six} shows that every finite
seed iterate occurs in $Y_0$.
Conversely, the iterates cover their straight inflated parents and hence
contain balls of arbitrarily large radius. Repetitivity places any given
patch of $Y_0$ inside a sufficiently large seed iterate. The two languages
coincide, so $Y=\SB$ after closure.
All constructions use relative positions and commute with translation.
\end{proof}

\begin{corollary}
\label{thm:recognizability}
The six-rule substitution on $\SB$ is recognizable: each tiling has a
unique marked parent, recoverable locally.
\end{corollary}
\begin{proof}
The unmarked geometry fixes the parent vertices and rhombs, and each
rhomb fixes its marked replacement rule. Existence follows from the phase
map on generic grids and from the same local-limit argument on the hull.
Every admitted parent satisfies these identities, so it is unique.
\end{proof}

\section{Conclusion}
The eleven local vertex patches identify the parent vertices of an
unmarked Anti-Penrose tiling. Matching the child patches to the six
replacement rules then recovers the Seabed markings. Erasure and recovery are inverse
local maps on the whole regular-limit hull, including its singular members.
Thus Seabed provides a recognizable substitution description of the
Anti-Penrose geometry.

\bibliographystyle{unsrt}
\bibliography{references}

\begin{thebibliography}{1}

\bibitem{Penrose1979}
Roger Penrose.
\newblock Pentaplexity: A class of non-periodic tilings of the plane.
\newblock {\em The Mathematical Intelligencer}, 2(1):32--37, 1979.

\bibitem{deBruijn1981}
N.~G. de~Bruijn.
\newblock Algebraic theory of {Penrose's} non-periodic tilings of the plane.
  {I}, {II}.
\newblock {\em Indagationes Mathematicae}, 43(1):39--66, 1981.

\bibitem{GrimmSchreiber2002}
Uwe Grimm and Michael Schreiber.
\newblock Aperiodic tilings on the computer.
\newblock In Jens-Boie Suck, Michael Schreiber, and Peter H{\"a}ussler,
  editors, {\em Quasicrystals: An Introduction to Structure, Physical
  Properties and Applications}, pages 49--66. Springer, Berlin, 2002.

\bibitem{Lutfalla2021}
Victor~Hubert Lutfalla.
\newblock {\em Substitution discrete planes}.
\newblock PhD thesis, Universit{\'e} Sorbonne Paris Nord, 2021.

\bibitem{BaakeSchlottmannJarvis1991}
Michael Baake, Martin Schlottmann, and Peter~D. Jarvis.
\newblock Quasiperiodic tilings with tenfold symmetry and equivalence with
  respect to local derivability.
\newblock {\em Journal of Physics A: Mathematical and General},
  24(19):4637--4654, 1991.

\end{thebibliography}
\end{document}